\documentclass[11pt,reqno]{amsart}

\usepackage{amsmath,amssymb,amsfonts,amscd,amsthm}
\usepackage{mathrsfs,mathabx}
\usepackage{latexsym,graphicx,cite,cases,array,booktabs}

\newtheorem{theorem}{Theorem}[section]

\newtheorem{corollary}[theorem]{Corollary}
\newtheorem{lemma}[theorem]{Lemma}

\numberwithin{equation}{section}

\title[an inverse source problem]{Stability for an inverse source problem of the biharmonic operator}

\author[P. Li]{Peijun Li}
\address{Department of Mathematics, Purdue University, West Lafayette, Indiana
47907, USA}
\email{lipeijun@math.purdue.edu}

\author[X. Yao]{Xiaohua Yao}
\address{School of Mathematics and Statistics, China Central Normal University,
Wuhan, Hubei, China}
\email{yaoxiaohua@mail.ccnu.edu.cn}

\author[Y. Zhao]{Yue Zhao}
\address{School of Mathematics and Statistics, China Central Normal University,
Wuhan, Hubei, China}
\email{zhaoy@mail.ccnu.edu.cn}

\thanks{The research of PL is supported in part by the NSF grant DMS-1912704. The research of XY is supported in part by NSFC (No. 11771165). The research of YZ is supported in part by NSFC (No. 12001222).}

\begin{document}

\begin{abstract}
In this paper, we study for the first time the stability of the inverse source problem for the biharmonic operator with a compactly supported potential in $\mathbb R^3$. Firstly, to connect the boundary data with the unknown source, we shall consider an eigenvalue problem for the bi-Schr$\ddot{\rm o}$dinger operator $\Delta^2 + V(x)$ on a ball which contains the support of the potential $V$. We prove a Weyl-type law for the upper bounds of spherical normal derivatives of both the eigenfunctions $\phi$ and their Laplacian $\Delta\phi$ corresponding to the bi-Schr$\ddot{\rm o}$dinger operator. This type of upper bounds was proved by Hassell and Tao for the Schr$\ddot{\rm o}$dinger operator. Secondly, we investigate the meromorphic continuation of the resolvent of the bi-Schr$\ddot{\rm o}$dinger operator and prove the existence of a resonance-free region and an estimate of $L^2_{\rm comp} - L^2_{\rm loc}$ type for the resolvent. As an application, we prove a bound of the analytic continuation of the data from the given data to the higher frequency data. Finally, we derive the stability estimate which consists of the Lipschitz type data discrepancy and the high frequency tail of the source function, where the latter decreases as the upper bound of the frequency increases.
\end{abstract}


\maketitle

\section{Introduction}

Consider the three-dimensional scattering problem for the biharmonic operator
\begin{align}\label{main_eq}
H u(x, \kappa) - \kappa^4u(x, \kappa) = f(x), \quad x \in \mathbb R^3,
\end{align}
where $H := \Delta^2 + V$, $\Delta$ is the Laplacian and  $V(x)$ is the potential, $\kappa>0$ is the wavenumber, and $f$ is the real-valued source term. We assume that $V(x)\in C_c^\infty(\mathbb R^3), V(x)\geq 0$ and both $f$ and $V$ have a compact support contained in $B_R = \{x\in\mathbb R^3 ~:~ |x|\leq R\}$, where $R>0$ is a constant. Let $\partial B_R$ be the boundary of $B_R$. An analogue of the Sommerfeld radiation condition is imposed to ensure the well-posedness of the problem (cf. \cite{ts_ipi}):
\begin{equation}\label{src}
\lim_{r\to\infty}r (\partial_r u-\mathrm{i}\kappa u)=0,\quad \lim_{r\to\infty}r (\partial_r (\Delta u)-\mathrm{i}\kappa (\Delta u))=0
\end{equation}
uniformly in all directions $\hat{x} = x/|x|$ with $r = |x|$. This paper is concerned with the inverse problem of determining the source function $f$ from the boundary measurements of $u(x, \kappa), \Delta u(x, \kappa)$ on $\partial B_R$ corresponding to the wavenumber $\kappa$ given in a finite interval.

The inverse scattering problems have played a fundamental role in diverse scientific areas such as radar and sonar, geophysical exploration, and medical imaging. The inverse problems for biharmonic operators have important applications in the study of elasticity and the theory of vibrations of beams, e.g., the beam equation \cite{GGS}, the hinged plate configurations \cite{GGS}, and the scattering by grating stacks \cite{MMM}. Compared with the inverse scattering problems for acoustic, elastic, and electromanetic waves, the inverse scattering problems for the biharmonic operators are much less studied. In fact, not only the increase of the order leads to the failure of the methods which work for the second order equations, but the properties of the solutions themselves become more involved \cite{Mayboroda}.
We refer to \cite{AP, Iwasaki, THS, ts_ipi} for the inverse scattering problems of higher order operators. The inverse boundary value problems for bi- and poly-harmonic operators can be found in \cite{Assy, CH, isakov, KLU, LKU, Yang}. The available results are mainly concerned with the inverse problem of determining the first order perturbation of the  form $A(x)\cdot \nabla + q(x)$ of the bi- and poly-harmonic operators by using either the far-field pattern or the Dirichlet-to-Neumann map on the boundary. A numerical study can be found in \cite{Xu} for an inverse random source for the biharmonic equation. To the best of our knowledge, the uniqueness and stability are open on the inverse source problem for the biharmonic operators.
 
 In general, it is known that there is no uniqueness for the inverse source problems at a fixed frequency. For example, if the source term $f:=H\varphi - \kappa^4\varphi$ where $\varphi\in C_0^\infty(B_R)$, it is easy to know that the uniqueness does not hold in this case. Computationally, a more serious issue is the lack of stability, i.e., a small variation of the data might lead to a huge error in the reconstruction. Hence it is crucial to study the stability of the inverse source problems. Recently, it has been realized that the use of multifrequency data is an effective approach to overcome the difficulties of non-uniqueness and instability which are encountered at a single frequency. The first stability result was obtain in \cite{blt} for the inverse source problem of the Helmholtz equation by using multifrequency data. The increasing stability was studied for the inverse source problems of the acoustic, elastic and electromagnetic wave equations \cite{blz, CIL, ei-18, ei-19, li2017increasing,
 LZZ}. A topic review can be found in \cite{BLLT-IP-15} on the general inverse scattering problems with multifrequency. 
 
Motivated by \cite{LZZ}, we intend to study the stability on the inverse source problem for the perturbed biharmonic operators by using multi-wavenumber data. We consider an eigenvalue problem for the biharmonic operator with a zeroth order perturbation and deduce an integral equation, which connects
the scattering data $u(x, \kappa)\vert_{\partial B_R}, \Delta u(x, \kappa)\vert_{\partial B_R}$ and the unknown source function $f$. Then we study the corresponding
resolvent of the biharmonic operator to obtain a resonance-free region of the data with respect to the complex wavenumber $\kappa$ and the bound of the analytic continuation of the data from the given data to the higher wavenumber data. It should be pointed out that in this work we obtain a resonance-free region for the resolvent and prove the resolvent estimate in this region. Then the well-posedness of the direct scattering problem follows. Furthermore, the results on the resolvent play a crucial role in the study of the inverse scattering problem, and they are also interesting in themselves.  The stability estimate consists of the Lipschitz type of data discrepancy and the high wavenumber tail of the source function. The latter decreases as the wavenumber of the data increases, which implies that the inverse problem is more stable when the higher wavenumber data is used. We also mention that only the Dirichlet data is required for the analysis.

The paper is organized as follows. In section \ref{without}, we show the increasing stability of the inverse source problem for the biharmonic operator without the zeroth order perturbation. Section \ref{with} is devoted to the general case where the biharmonic operator has a nontrivial potential. In both sections, the resolvent is studied for the corresponding biharmonic operator and its resonance-free region and upper bound are obtained, which lead to the well-posedness of the direct scattering  problems and are crucial for the stability analysis of the inverse source problem by using discrete multi-wavenumber data. The paper concludes with some general remarks in section \ref{con}.  

\section{Stability without zeroth order perturbation}\label{without}

In this section, we discuss the well-posedness of the direct and inverse problems for the biharmonic operator without the zeroth order perturbation. 

\subsection{Resolvent estimate}

We begin with the resolvent estimate of $H_0 := \Delta^2$, which is the biharmonic operator without the zeroth order perturbation. Clearly, the operator $H_0 $ is self-adjoint on $L^2(\mathbb R^3)$ with the Sobolev domain $H^4(\mathbb R^3)$. It follows from the Fourier transform that 
\begin{align*}
\widehat{H_0f}(\xi) = |\xi|^4\hat{f}(\xi), \quad f \in H^4(\mathbb R^3),
\end{align*}
which immediately  deduces the spectrum of $H_0$:
\[
\sigma(H_0) = \{z=|\xi|^4: \xi\in\mathbb R^3\} = [0, +\infty).
\]
Hence the resolvent $(H_0 - z)^{-1}$ of $H_0$ is analytic for $z\in\mathbb{C}\backslash [0, +\infty)$ in the uniform operator topology of $\mathcal B(L^2, L^2)$, where $\mathcal B(L^2, L^2)$ denotes the set of all bounded operators of $L^2(\mathbb R^3)$.

It is clear to note that the holomorphic map $F : \lambda \rightarrow \lambda^4$ takes the first quadrant $$\Sigma := \Big\{\lambda\in\mathbb C: 0<{\rm arg}\lambda<\frac{\pi}{2}\Big\}$$ bijectively onto the set $\sigma(H_0) = \mathbb{C}\backslash [0, +\infty)$, which is shown in Figure \ref{map}. 

\begin{figure}
\centering
\includegraphics[width=1.0\textwidth]{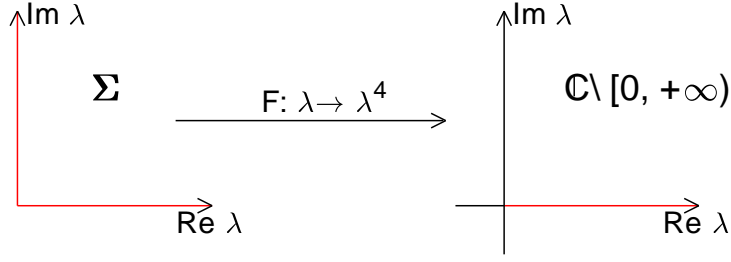}
\caption{The holomorphic map of $F$.}
\label{map}
\end{figure}

Let $z = \lambda^4$. The family of operators $R_0(\lambda) := (H_0 - \lambda^4)^{-1}$ are analytical on the first quadrant $\Sigma$ and satisfy the resolvent estimate
\begin{align*}
\|R_0(\lambda)\|_{L^2(\mathbb R^3)\to L^2(\mathbb R^3)}\leq \frac{1}{{\rm dist}(\lambda^4, [0, \infty))},\quad \lambda\in \Sigma.
\end{align*}
Recall that the operator $(-\Delta - \lambda^2)^{-1}$ is well-defined on $L^2(\mathbb R^3)$ for $\Im\lambda >0$ via the explicit expression
\begin{align*}
(-\Delta - \lambda^2)^{-1}(f) = \int_{\mathbb R^3} \frac{e^{{\rm i}\lambda|x-y|}}{4\pi|x-y|} f(y){\rm d}y.
\end{align*}
For $\lambda\in\Sigma$, using the identity 
\begin{align*}
R_0(\lambda) = (\Delta^2 - \lambda^4)^{-1} = \frac{1}{2\lambda^2}[ (-\Delta - \lambda^2)^{-1} - (-\Delta^2 + \lambda^2)^{-1} ],
\end{align*}
we have 
\begin{align*}
(R_0(\lambda)f)(x) = \int_{\mathbb R^3} R_0(x, y, \lambda) f(y) {\rm d}y,
\end{align*}
where 
\begin{align}\label{kernel}
R_0(\lambda, x, y) = \frac{1}{2\lambda^2} \Big( \frac{e^{{\rm i}\lambda|x-y|}}{4\pi|x-y|} - \frac{e^{-\lambda|x-y|}}{4\pi|x-y|} \Big), \quad \lambda\in\Sigma.
\end{align}

It is easy to verify that the kernel $R_0(\lambda, x, y)$ satisfies the radiation condition \eqref{src} for fixed $x$ or $y$ and $\lambda>0$. Furthermore, we can see from \eqref{kernel} that for fixed $x$ and $y$, $R_0(\lambda)$ is a meromorphic function of $\lambda$ on $\mathbb{C}$ and defines an operator $C_0^\infty(\mathbb R^3) \to C^\infty(\mathbb R^3)$, which is unbounded on $L^2(\mathbb R^3)$ for $\Im\lambda\leq 0$ or $\Re\lambda\geq 0.$
However, if we consider $R_0(\lambda)$ as an operator mapping $L^2_{\rm comp}(\mathbb R^3)$ onto $L^2_{\rm loc}(\mathbb R^3)$ in the sense that for any $\rho\in C_0^\infty(\mathbb R^3)$,  the operator $\rho R_0(\lambda) \rho: L^2(\mathbb R^3)\to L^2(\mathbb R^3)$ is bounded, then the operator $R_0(\lambda)$ can be extended into a meromorphic family of operators for all $\lambda \in \mathbb{C}$. 

The following theorem concerns a resonance-free region and an estimate for the resolvent $R_0(\lambda)$, which play a crucial role in the stability analysis for the inverse problem. Hereafter, the notation $a\lesssim b$ stands for $a\leq Cb,$ where $C>0$ is a generic constant which may change step by step in the proofs.

\begin{theorem}\label{free_estimate}
The resolvent operator $R_0(\lambda)$ for $\lambda\in\Sigma$ has the following estimate:
\begin{align}\label{spectral1}
\|R_0(\lambda)\|_{L^2(\mathbb R^3)\to L^2(\mathbb R^3)}\leq \frac{1}{|\lambda|^2(\Im\lambda)(\Re\lambda)}, \quad \lambda\in\Sigma.
\end{align}
Moreover, the operator $R_0(\lambda)$ can be extended into an analytic family of operators for all $\lambda\in\mathbb{C}\backslash\{0\}$ as
\begin{align*}
R_0(\lambda): L^2_{\rm comp}(\mathbb R^3)\to L^2_{\rm loc}(\mathbb R^3),
\end{align*}
such that for each $\rho\in C_0^\infty(\mathbb R^3)$ with ${\rm supp}(\rho)\subset B_R$ and $\lambda\neq 0$
\begin{align}\label{free}
\|\rho R_0(\lambda) \rho\|_{L^2(B_R)\rightarrow H^j(B_R)}\lesssim |\lambda|^{-2}(1+\lambda^2)^{\frac{j}{2}} \big(e^{2R (\Im\lambda)_-} + e^{2R (\Re\lambda)_-}\big), \quad j=0, 1, 2, 3, 4,
\end{align}
where $t_{-}:=\max\{-t,0\}$. 
\end{theorem}

\begin{proof}
First we show \eqref{spectral1}. Following the identity 
\[
\lambda^4 = (\Re\lambda)^4 + (\Im\lambda)^4 - 6(\Re\lambda)^2(\Im\lambda)^2 + 
4{\rm i}(\Re\lambda)(\Im\lambda)((\Re\lambda)^2 - (\Im\lambda)^2),
\]
we have 
\begin{align*}
{\rm dist}&(\lambda^4, [0, \infty))=\\
&\begin{cases}
4(\Re\lambda)(\Im\lambda)|(\Re\lambda)^2 - (\Im\lambda)^2| &\quad\text{if} ~  (\Re\lambda)^4 + (\Im\lambda)^4\geq 
6(\Re\lambda)^2(\Im\lambda)^2;\\
|\lambda|^4 & \quad\text{if} ~ (\Re\lambda)^4 + (\Im\lambda)^4<6(\Re\lambda)^2(\Im\lambda)^2.
\end{cases}
\end{align*}
Note that when $ (\Re\lambda)^4 + (\Im\lambda)^4\geq 6(\Re\lambda)^2(\Im\lambda)^2$, one has
\[
|\lambda|^4 = ((\Re\lambda)^2 + (\Im\lambda)^2)^2 \leq 2((\Re\lambda)^2 - (\Im\lambda)^2)^2,
\]
which yields $|\lambda|^2\leq \sqrt{2}|(\Re\lambda)^2 - (\Im\lambda)^2|$. On the other hand, when $(\Re\lambda)^4 + (\Im\lambda)^4<6(\Re\lambda)^2(\Im\lambda)^2$, we have
\[
|\lambda|^2 = (\Re\lambda)^2 + (\Im\lambda)^2 \geq 2|\lambda|^2 (\Re\lambda)(\Im\lambda), \quad \lambda\in\Sigma,
\]
which gives \eqref{spectral1}.

Next is to show the analytical extension. Let $\rho\in C_0^\infty(\mathbb R^3)$, then for each $\lambda\neq 0$,  we can define the operator $\rho R_0(\lambda)\rho$ by
\begin{align}\label{kernel2}
(\rho R_0(\lambda)\rho f)(x) = \frac{1}{8\pi\lambda}\int_{\mathbb R^3}\rho(x)\frac{e^{{\rm i}\lambda|x-y|} - e^{-\lambda|x-y|}}{|x-y|}\rho(y)f(y){\rm d}y.
\end{align}
It can be verified that $\rho R_0(\lambda)\rho$ is bounded on $L^2(\mathbb R^3)$ and satisfies the estimate
\begin{align}\label{free2}
\|\rho R_0(\lambda)\rho\|_{L^2(B_R)\rightarrow L^2(B_R)} &\leq \frac{1}{8\pi|\lambda|^2}\Big(e^{2R(\Im\lambda)_-} + e^{2R(\Re\lambda)_-}\Big)\times\notag\\
&\qquad \Big(\int_{B_R}\int_{B_R} \rho^2(x)\frac{1}{|x-y|^2}\rho^2(y){\rm d}x{\rm d}y\Big)^{1/2}\notag\\
&\leq C |\lambda|^{-2}(e^{2R(\Im\lambda)_-} + e^{2R(\Re\lambda)_-}),
\end{align}
which actually implies that the operator $\rho R_0(\lambda)\rho$ belongs to the Hilbert--Schmidt class. Moreover, for any $f, g\in L^2(\mathbb R^3)$, by the explicit expression \eqref{kernel2} of $\rho R_0(\lambda)\rho$, it is easy to prove that the function $I_0(\lambda):=\langle \rho R_0(\lambda)\rho f, g\rangle_{L^2(\mathbb R^3)}$ is an analytic function in $\mathbb{C}\backslash\{0\}$ and $\lambda = 0$ is the only simple pole. Consequently, $\rho R_0(\lambda)\rho$ is an analytic family of compact operators for $\lambda\in\mathbb{C}\backslash\{0\}$.

It suffices to prove the case $j=4$ in order to prove \eqref{free}. By the standard elliptic estimate, one has for any $\Omega\subset\subset W \subset \mathbb{R}^3$ that 
\begin{align*}
\|u\|_{H^4(\Omega)}\leq C\big(\|u\|_{L^2(W)} + \|\Delta^2u\|_{L^2(W)}\big).
\end{align*}
Taking $\tilde{\rho}\in C^\infty_0(B_R)$ such that $\tilde{\rho} = 1$ near the support of $\rho$, we obtain 
\[
\|\rho u\|_{H^4(B_R)}\leq C\big(\|\tilde{\rho}u\|_{L^2(\mathbb{R}^3)} + \|\tilde{\rho}\Delta^2u\|_{L^2(\mathbb{R}^3)}\big).
\]
Thus letting $u = R_0(\lambda)(\rho f), f\in L^2(B_R)$ gives 
\[
\|\rho R_0(\lambda)(\rho f)\|_{H^4(B_R)}\leq C\big(\|\tilde{\rho}R_0(\lambda)(\rho f)\|_{L^2(\mathbb R^3)} + \|\tilde{\rho}\Delta^2(R_0(\lambda)(\rho f))\|_{L^2(\mathbb R^3)}\big).
\]
Noting that 
\begin{align*}
\|\tilde{\rho}\Delta^2(R_0(\lambda)(\rho f))\|_ {L^2(\mathbb R^3)}&=\| \rho f + \tilde{\rho}\lambda^4R_0(\lambda)(\rho f)\|_{L^2(\mathbb R^3)}\\& \lesssim (1+\lambda^2) \big(e^{2R(\Im\lambda)_-} + e^{2R(\Re\lambda)_-}\big)\|f\|_{L^2(B_R)},
\end{align*}
we hence obtain that
\[
\|\rho R_0(\lambda)\rho\|_{L^2(B_R)\rightarrow H^4(B_R)}\lesssim |\lambda|^{-2} (1+ \lambda^2)^2 
\big(e^{2R(\Im\lambda)_-} + e^{2R(\Re\lambda)_-}\big).
\]
Finally, the cases for $j = 1, 2, 3$ follow by the interpolation between $j=0$ and $j=4$.
\end{proof}

It follows from Theorem \ref{free_estimate} that the scattering problem \eqref{main_eq}--\eqref{src} has a unique solution for all the positive wavenumbers when $V(x)\equiv 0$, which is stated in the following result.

\begin{corollary}
Let $V(x)\equiv 0$. For any $\kappa>0$, the scattering problem \eqref{main_eq}--\eqref{src} admits a unique solution $u\in H^4(B_R)$ such that 
\[
\|u\|_{H^4(B_R)}\lesssim \|f\|_{L^2(B_R)}.
\]
\end{corollary}

\subsection{Inverse problem}

In this section, we discuss the uniqueness and stability of the inverse problem without zeroth order perturbation, i.e., $V(x)\equiv 0$. 

First, we study the spectrum of the operator $H_0$ with the Navier boundary condition. Let $\{\lambda_j, \varphi_j\}_{j=1}^\infty$ be the positive increasing eigenvalues and eigenfunctions of $H_0$ in $B_R$, where $\varphi_j$ and $\lambda_j$ satisfy
\[
\begin{cases}
\Delta^2 \varphi_j(x)=\lambda_j\varphi_j(x)&\quad\text{in }B_R,\\
\Delta \varphi_j(x) = \varphi_j(x)=0&\quad\text{on }\partial B_R.
\end{cases}
\]
In fact, we can take $\{\lambda_j^{1/2}, \varphi_j\}_{j=1}^\infty$ to be the spectrum of the Laplacian operator such that
\[
\begin{cases}
-\Delta \varphi_j(x)=\lambda_j^{1/2}\varphi_j(x)&\quad\text{in }B_R,\\
\varphi_j(x)=0&\quad\text{on }\partial B_R,
\end{cases}
\]
where the eigenfunctions $\{\varphi_j\}_{j=1}^\infty$ form a complete basis in $L^2(B_R)$. Assume that $\varphi_j$ is normalized such that
\[
\int_{B_R}|\varphi_j(x)|^2{\rm d}x=1.
\] 
Consequently, we obtain the spectral decomposition of $f$:
\[
f(x)=\sum_{j=1}^\infty f_j\varphi_j(x),
\]
where
\[
f_j=\int_{B_R}f(x)\bar{\varphi}_j(x){\rm
d}x.
\]
It is clear that 
\begin{align}\label{energy}
 \|f\|^2_{L^2(B_R)} = \sum_{j} |f_j|^2.
\end{align}

The following lemma gives a link between the values of an analytical function for small and large arguments. The proof can be found in \cite[Lemma 3.2]{CIL}. 

\begin{lemma}\label{AC1}
  Denote $S=\{z=x+{\rm i}y\in\mathbb{C}: -\frac{\pi}{4}<{\rm arg}
z<\frac{\pi}{4}\}$. Let $J(z)$ be analytic in $S$ and continuous in $\bar{S}$
satisfying
\[
 \begin{cases}
  |J(z)|\leq\epsilon, & z\in (0, ~ K],\\
  |J(z)|\leq M, & z\in S,\\
  |J(0)|=0.
 \end{cases}
\]
Then there exits a function $\beta(z)$ satisfying
\[
 \begin{cases}
  \beta(z)\geq\frac{1}{2},  & z\in(K, ~ 2^{\frac{1}{4}}K),\\
  \beta(z)\geq \frac{1}{\pi}((\frac{z}{K})^4-1)^{-\frac{1}{2}}, & z\in
(2^{\frac{1}{4}}K, ~ \infty)
 \end{cases}
\]
such that
\[
|J(z)|\leq M\epsilon^{\beta(z)}\quad\forall\, z\in (K, ~ \infty).
\]
\end{lemma}

The following result concerns an estimate for the normal derivatives of the eigenfunctions on $\partial B_R$. The proof can be found in \cite[Lemma A.2]{LZZ}.

\begin{lemma}\label{eigenfunction_est2}
The following estimate holds:
\begin{align*}
\|\partial_\nu \varphi_j\|_{L^2(\partial B_R)}\leq C\kappa_j,
\end{align*}
where the positive constant $C$ is independent of $j$. Moreover, the following Weyl-type inequality holds for the Dirichlet eigenvalues $\{\lambda_n\}_{n=1}^\infty$:
\begin{align*}
E_1 n^{4/3}\leq \lambda_n\leq E_2 n^{4/3},
\end{align*}
where $E_1$ and $E_2$ are two positive constants independent of $n$.
\end{lemma}

Denote $\kappa_j^4 = \lambda_j$. Let
$u(x, \kappa_j)$ be the solution to
\eqref{main_eq}--\eqref{src} with $\kappa = \kappa_j$. 

\begin{lemma}\label{fj}
The following estimate holds:
\[
|f_j|^2 \lesssim \kappa_j^4 \|u(x,\kappa_j)\|^2_{L^2(\partial B_R)}  + \kappa_j^2 \|\Delta u(x,\kappa_j)\|^2_{L^2(\partial B_R)}
\]
for $j=1, 2, 3, \cdot\cdot\cdot$.
\end{lemma}

\begin{proof}

Multiplying both sides of \eqref{main_eq} by $\bar{\varphi}_j$ and using
the integration by parts yield
\begin{align*}
\int_{B_R} f(x)\bar{\varphi}_j(x) {\rm d}x &= \int_{\partial B_R}\left(\partial_\nu (\Delta u(x, \kappa_j)) \bar{\varphi}_j - \Delta u(x, \kappa_j)\partial_\nu \bar{\varphi}_j\right){\rm d}s\\
 &\quad + \int_{\partial B_R} \left(\partial_\nu u(x, \kappa_j) \Delta \bar{\varphi}_j - u(x, \kappa_j)\partial_\nu (\Delta \bar{\varphi}_j)\right){\rm d}s.
\end{align*}
Noting $\Delta \bar{\varphi}_j =  \bar{\varphi}_j = 0$ and $\partial_\nu (\Delta \bar{\varphi}_j) = -\kappa_j^2\partial_\nu\bar{\varphi}_j$ on $\partial B_R$, we obtain 
\begin{align*}
\int_{B_R} f(x)\bar{\varphi}_j(x) {\rm d}x = -\int_{\partial B_R} \Delta u(x, \kappa_j)\partial_\nu \bar{\varphi}_j{\rm d}s + \kappa_j^2\int_{\partial B_R}u(x, \kappa_j)\partial_\nu \bar{\varphi}_j{\rm d}s.
\end{align*}
The proof is completed by using Lemma \ref{eigenfunction_est2} and the Schwartz
inequality. 
\end{proof}

Let $\varepsilon$ be a small positive constant such that $\varepsilon<\kappa_1$ and denote $\Omega_{\varepsilon}:= \mathbb R^3\backslash \overline{B_{\varepsilon}}$. The following lemma gives the analytic continuation of the data from small wavenumber to large wavenumber. The proof is based on the crucial result in Theorem \ref{free_estimate}. 

\begin{lemma}\label{ac_1}
Let $f$ be a real-valued function and $\|f\|_{L^2(B_R)}\leq Q$.
Then for any two positive constants $A$ and $A_1$ satisfying
$A_1 = A + \varepsilon<\kappa_1$, there exists a function $\beta(\kappa)$ such that 
\[
\kappa^4 \|u(x,\kappa)\|^2_{L^2(\partial B_R)}  + \kappa^2 \|\Delta u(x,\kappa)\|^2_{L^2(\partial B_R)} \lesssim
Q^2e^{6R\kappa}\epsilon_1^{2\beta(\kappa)}\quad \forall\, \kappa\in (A_1, +\infty),
\]
where 
\begin{align*}
\epsilon^2_1 :&={\rm sup}_{\kappa \in (\varepsilon, A_1)} \Big(\kappa^4 \|u(x,\kappa)\|^2_{L^2(\partial B_R)}  + \kappa^2 \|\Delta u(x,\kappa)\|^2_{L^2(\partial B_R)}\Big)
\end{align*}
and $\beta(\kappa)$ satisfies
\begin{align}\label{mu1}
 \begin{cases}
  \beta(\kappa)\geq\frac{1}{2},  & \kappa\in(A_1, ~ \varepsilon + 2^{\frac{1}{4}}A),\\
  \beta(\kappa)\geq \frac{1}{\pi}((\frac{\kappa - \varepsilon}{A})^4-1)^{-\frac{1}{2}}, & \kappa\in
(\varepsilon + 2^{\frac{1}{4}}A, ~ \infty). 
 \end{cases}
\end{align}

\end{lemma}

\begin{proof}

Let 
\[
I(\kappa):=\int_{\partial
B_R}\left(\kappa^2 u(x,\kappa)u(x,{\rm i}\kappa) + \kappa^4 \Delta u(x,\kappa)\Delta u(x,{\rm i}\kappa)\right)
{\rm d}s,\quad \kappa\in \mathbb C.
\] 
Since $f(x)$ is a real-valued function,  it holds that $\overline{u(x,\kappa)}=u(x,{\rm i}\kappa)$ and  $\overline{\Delta u(x,\kappa)}=\Delta u(x,{\rm i}\kappa)$ for $\kappa\in\mathbb{R}^+$. Thus, we have 
\[
I(\kappa)=\kappa^2\|u(x,\kappa)\|^2_{L^2(\partial B_R)} + \kappa^4\|\Delta u(x,\kappa)\|^2_{L^2(\partial B_R)}, \quad \kappa\in\mathbb R^+.
\]
It follows from Theorem \ref{free_estimate} that $I(\kappa)$ is analytic in the sector domain $S_\varepsilon: =\{z: -\frac{\pi}{4}<{\rm arg}(z - \varepsilon)<\frac{\pi}{4}\}$.
By the estimate \eqref{free} in Theorem \ref{free_estimate}, we have for $\kappa\in S_\varepsilon$ that 
\begin{align*}
&|\kappa|\|u(x,\kappa)\|_{L^2(\partial B_R)} + |\kappa|^2 \|\Delta u(x,\kappa)\|_{L^2(\partial B_R)}\\
&\leq  |\kappa|\|u(x,\kappa)\|_{H^{1/2}(\partial B_R)} + |\kappa|^2 \|\Delta u(x,\kappa)\|_{H^{3/2}(\partial B_R)}\\
&\leq (|\kappa| + |\kappa|^2) \|u\|_{H^4(\mathbb R^3)}\leq  e^{3R|\kappa|}\|f\|_{L^2(B_R)}.
\end{align*}
Since
\begin{align*}
|I(\kappa)|&\leq |\kappa| \|u(x,\kappa)\|_{L^2(\partial B_R)} |\kappa|
\|u(x,-\kappa)\|_{L^2(\partial B_R)}  \\
&\quad + |\kappa|^2 \|\Delta u(x,\kappa)\|_{L^2(\partial B_R)} |\kappa|^2
\|\Delta u(x,-\kappa)\|_{L^2(\partial B_R)}\\
&\lesssim e^{6R|\kappa|}\|f\|^2_{L^2(B_R)},\quad\kappa\in S_\varepsilon,
\end{align*}
we get
\[
|e^{-6R|\kappa|}I(\kappa)|\lesssim Q^2,\quad\kappa\in S_\varepsilon.
\]
An application of Lemma \ref{AC1} shows that there exists a function $\beta(\kappa)$ satisfying \eqref{mu1} such that
\[
\big| e^{-6R\kappa} I(\kappa)\big| \lesssim
Q^2\epsilon_1^{2\beta(\kappa)}\quad \forall\, \kappa\in (A_1, +\infty),
\]
which completes the proof.
\end{proof}

We state a simple uniqueness result for the inverse problem. 

\begin{theorem}
Let $f\in L^2(B_R)$ and $I\subset\mathbb R^+$ be an open interval. Then the source term can be uniquely determined by the multifrequency data $\{u(x, \kappa), \Delta u(x, \kappa): x\in\partial B_R, \kappa\in I\}$.
\end{theorem}

\begin{proof}
Let $u(x, \kappa) = \Delta u(x, \kappa) = 0$ for all $x\in\partial B_R$ and $\kappa\in I$.  It suffices to prove that $f (x) = 0$. Since by Theorem \ref{free_estimate} it holds that $u(x, \kappa)$ is analytic in the whole complex plane minus the origin, i.e., $\mathbb C\backslash\{0\}$, we obtain that $u(x, \kappa) = \Delta u(x, \kappa) = 0$ for all $\kappa\in\mathbb C\backslash\{0\}$. Hence we have $u(x, \kappa_j) = \Delta u(x, \kappa_j) = 0$ for all $ j = 1, 2, 3\cdots$. Then by \eqref{energy} and Lemma \ref{fj} we have $f = 0.$
\end{proof}

The following lemma is important in the stability analysis. The proof can be found in \cite[Lemma 3.4]{LZZ}.

\begin{lemma}\label{tail}
Let $f\in H^{n+1}(B_R)$ and $\|f\|_{H^{n+1}(B_R)}\leq Q$. The following estimate holds:  
\[
\sum_{j \geq s} |f_j|^2 \lesssim \frac{Q^2}{s^{\frac{2}{3}(n+1)}}.
\]
\end{lemma}

Define a real-valued function space
\[
\mathcal C_Q = \{f \in H^{n+1}(B_R):  \|f\|_{H^{n+1}(B_R)}\leq Q, ~ \text{supp}
f\subset B_R, ~ f: B_R \rightarrow \mathbb R \}.
\]
Now we are in the position to present the stability of the inverse source problem. Let $f\in
\mathcal C_Q$. The inverse source problem is to determine $f$ from the boundary data $u(x,\kappa), \Delta u(x,\kappa)$, $x\in\partial B_R$, $\kappa \in (\varepsilon, A_1) \cup \cup_{j=1}^N \kappa_j$, where $1\leq N \in \mathbb N$, $\varepsilon$ and $A_1$ are the constants given in Lemma \ref{ac_1}.

\begin{theorem}
Let $u(x,\kappa)$ be the solution of the scattering problem \eqref{main_eq}--\eqref{src} corresponding to the source $f\in \mathcal C_Q$. Then for sufficiently small $\epsilon_1$, the following estimate holds:
\begin{align}\label{stability_1}
\|f\|_{L^2( B_R)}^2  \lesssim \epsilon^2+\frac{Q^2}{\left(|\ln\epsilon_1|^{\frac{1}{9}}N^{\frac{5}{8}}
\right)^{\frac{2}{3}(2n+1)}},
\end{align}
where
\begin{align*}
\epsilon^2 &= \sum_{j=1}^N \kappa_j^4 \|u(x,\kappa_j)\|^2_{L^2(\partial B_R)}  + \kappa_j^2 \|\Delta u(x,\kappa_j)\|^2_{L^2(\partial B_R)},\\
\epsilon^2_1 &= {\rm sup}_{\kappa \in (\varepsilon,A_1)}\kappa^4 \|u(x,\kappa)\|^2_{L^2(\partial B_R)}  + \kappa^2 \|\Delta u(x,\kappa)\|^2_{L^2(\partial B_R)}.
\end{align*}
\end{theorem}

\begin{proof}

We can assume that $\epsilon_1<1$, otherwise the estimate is obvious. Let
\[
L=
 \begin{cases}
  [N^{\frac{3}{4}}|\ln\epsilon_1|^{\frac{1}{9}}], &
N^{\frac{3}{8}}<\frac{1}{2^{\frac{5}{6}}
\pi^{\frac{2}{3}}}|\ln\epsilon_1|^{\frac{1}{9}},\\
N,&N^{\frac{3}{8}}\geq \frac{1}{2^{\frac{5}{6}}
\pi^{\frac{2}{3}}}|\ln\epsilon_1|^{\frac{1}{9}}
 \end{cases}.
\]
Using Lemmas \ref{fj} and \ref{ac_1} lead to 
\begin{align*}
|f_j|^2 &\lesssim \kappa_j^4 \|u(x,\kappa_j)\|^2_{L^2(\partial B_R)}  + \kappa_j^2 \|\Delta u(x,\kappa_j)\|^2_{L^2(\partial B_R)}\\
&\lesssim Q^2e^{6R\kappa}\epsilon_1^{2\beta(\kappa)}\lesssim Q^2
e^{4\kappa}e^{-2\beta(\kappa)|\ln\epsilon_1|}\\
&\leq C Q^2
e^{4\kappa}e^{-\frac{2}{\pi}(\kappa^4-1)^{-\frac{1}{2}}|\ln\epsilon_1|}\lesssim
Q^2 e^{4\kappa-\frac{2}{\pi}\kappa^{-2}|\ln\epsilon_1|}\\
&\leq
CQ^2 e^{-\frac{2}{\pi}\kappa^{-2}|\ln\epsilon_1|(1-2\pi\kappa^3|\ln\epsilon_1|^{
-1})}\quad \forall\,\kappa\in (\varepsilon + 2^{\frac{1}{4}}A,~\infty).
\end{align*}
Hence we have
\begin{equation}\label{c2}
 |f_j|^2\lesssim Q^2
e^{-\frac{2}{\pi L^2}L^{-2}|\ln\epsilon_1|(1-2\pi L^3|\ln\epsilon_1|^{
-1})}\quad\forall~j\in (\varepsilon + 2^{\frac{1}{4}}A,~L].
\end{equation}
If $N^{\frac{3}{8}}<\frac{1}{2^{\frac{5}{6}}
\pi^{\frac{2}{3}}}|\ln\epsilon_1|^{\frac{1}{9}}$, then $2\pi
L^3|\ln\epsilon_1|^{-1}<\frac{1}{2}$ and
\begin{equation}\label{c3}
 e^{-\frac{2}{\pi}\frac{|\ln\epsilon_1|}{L^2}}\leq
e^{-\frac{2}{\pi}\frac{|\ln\epsilon_1|}{N^{\frac{3}{2}}|\ln\epsilon_1|^{\frac{
2}{9}}}}\leq
e^{-\frac{2}{\pi}\frac{|\ln\epsilon_1|^{\frac{7}{9}}}{N^\frac{3}{2}}}\leq
e^{-\frac{2}{\pi}\frac{2^5\pi^4
|\ln\epsilon_1|^{\frac{1}{9}} N^{\frac{9}{4}}}{N^{\frac{3}{2}}}}= e^{-64\pi^3
|\ln\epsilon_1|^{\frac{1}{9}}N^{\frac{3}{4}}}.
\end{equation}
Combining \eqref{c2} and \eqref{c3}, we obtain
\begin{align*}
  |f_j|^2\lesssim Q^2 e^{-32\pi^3
|\ln\epsilon_1|^{\frac{1}{9}} N^{\frac{3}{4}}}\quad\forall\kappa\in
(\varepsilon + 2^{\frac{1}{4}}A,~L].
\end{align*}
It is easy to note that
\[
 e^{-x}\leq \frac{(6(n+1)-3)!}{x^{3(2(n+1)-1)}}\quad\text{for} ~ x>0.
\]
We have
\[
 |f_j|^2\lesssim
Q^2\frac{1}{\left(\frac{|\ln\epsilon_1|^{\frac{1}{3}}N^{\frac{9}{4}}}{(6n-3)^3}\right)^{2n + 1}},
\quad j=1, \dots, L.
\]
Consequently, we obtain
\begin{align*}
|f_j|^2 &\lesssim
Q^2\frac{L}{\left(\frac{|\ln\epsilon_1|^{\frac{1}{3}}N^{\frac{9}{4}}}{(6n-3)^3}\right)^{2n+1}}
\lesssim Q^2\frac{N^{\frac{3}{4}}|\ln\epsilon_1|^{\frac{1}{9}}}{\left(\frac{
|\ln\epsilon_1|^{\frac{1}{3}}N^{\frac{9}{4}}}{(6n-3)^3}\right)^{2n+1}}\\
&\lesssim Q^2
\frac{1}{\left(\frac{|\ln\epsilon_1|^{\frac{2}{9}}N^{\frac{3}{2}}}{(6n-3)^3}\right)^{2n+1}}
\lesssim Q^2
\frac{1}{\left(\frac{|\ln\epsilon_1|^{\frac{1}{9}}N^{\frac{3}{2}}}{(6n-3)^3}\right)^{2n + 1}},
\end{align*}
where we have used $|\ln\epsilon_1|>1$ when $N^{\frac{3}{8}}<\frac{1}{2^{\frac{5}{6}}
\pi^{\frac{2}{3}}}|\ln\epsilon_1|^{\frac{1}{9}}$.
If $N^{\frac{3}{8}}<\frac{1}{2^{\frac{5}{6}}
\pi^{\frac{2}{3}}}|\ln\epsilon_1|^{\frac{1}{9}}$, we also have
\begin{align*}
\frac{1}{\left(\bigl[|\ln\epsilon_1|^{\frac{1}{9}}N^{\frac{3}{4}}\bigr]+1\right)^{2n+1}}
\leq
\frac{1}{\left(|\ln\epsilon_1|^{\frac{1}{9}}N^{\frac{3}{4}}\right)^{2n+1}}.
\end{align*}
If $N^{\frac{3}{8}}\geq \frac{1}{2^{\frac{5}{6}}
\pi^{\frac{2}{3}}}|\ln\epsilon_1|^{\frac{1}{9}}$, then $L=N$. It follows from
Lemma \ref{fj} that
\[
 \sum_{j=1}^L |f_j|^2\lesssim\epsilon^2.
\]
Combining the above estimates, we obtain
\begin{align*}
\sum_{j=1}^\infty |f_j|^2 &\lesssim
\epsilon^2+ 
\frac{Q^2}{\left(\frac{|\ln\epsilon_1|^{\frac{1}{9}}N^{\frac{3}{2}}}{(6n-3)^3}\right)^{2n+1}}\\
&\quad+ \frac{Q^2}{\bigl(|\ln\epsilon_1|^{\frac{1}{9}}N^{\frac{3}{4}}
\bigr)^{\frac{2}{3}(2n+1)}}+ 
\frac{Q^2}{\left(|\ln\epsilon_1|^{\frac{1}{9}}N^{\frac{5}{8}}
\right)^{\frac{2}{3}(2n+1)}}\\
&\lesssim \epsilon^2+\frac{Q^2}{\left(|\ln\epsilon_1|^{\frac{1}{9}}N^{\frac{5}{8}}
\right)^{\frac{2}{3}(2n+1)}},
\end{align*}
which completes the proof.
\end{proof}

It can be observed that the stability estimate \eqref{stability_1} consists of two parts: the data discrepancy and
the high frequency tail. The former is of the Lipschitz type. The latter decreases as $N$ increases which makes the problem an almost Lipschitz stability. The result reveals that the problem becomes more stable when higher
frequency data is used.

\section{Stability with zeroth order perturbation}\label{with}

In this section, we discuss the well-posedness of the direct and inverse problems for the biharmonic operator with a general potential, i.e., $H = \Delta^2 + V(x)$ where $V(x)\in L^\infty_{\rm comp}(\mathbb R^3, \mathbb{C})$.

\subsection{Resolvent estimate}

Denote by $T: L^2_{\rm comp}(\mathbb R^3)\rightarrow L^2_{\rm loc}(\mathbb R^3)$ an operator $T$ such that for any $\rho\in C_0^\infty(\mathbb R^3)$ the operator $\rho T \rho: L^2(\mathbb R^3)\rightarrow L^2(\mathbb R^3)$ is bounded. Below is the analytic Fredholm theory. The result is classical and the proof may be found in many references, e.g., \cite[Theorem 8.26]{CK}.

\begin{theorem}
Let $D$ be a domain in $\mathbb C$ and let $\mathcal{A}: D \rightarrow \mathcal{L}(X)$ be an operator
valued analytic function such that $\mathcal{A}(z)$ is compact for each $z\in D$. Then either
\begin{itemize}

\item[(a)] $(I - \mathcal{A}(z))^{-1}$ does not exist for any $z\in D$ or

\item[(b)] $(I - \mathcal{A}(z))^{-1}$ exists for all $z\in D\backslash S$ where $S$ is a discrete subset of $D$.

\end{itemize}
Here $X$ is a Banach space and $\mathcal{L}(X)$ denotes the Banach space
of bounded linear operators mapping the Banach space $X$ into itself.
\end{theorem}

The following theorem gives a meromorphic continuation of the resolvent of the biharmonic operator $H$. 

\begin{theorem}
The resolvent
\[
R_V= (\Delta^2 + V - \lambda^4)^{-1}: L^2(\mathbb R^3)\rightarrow L^2(\mathbb R^3)
\]
is a meromorphic family of operators with a finite number of poles on the first quadrant $\Sigma$. Moreover, the family $R_V(\lambda)$ can be extended into a meromorphic family of the whole complex plane $\mathbb C$ in the sense that 
$
\rho R_V(\lambda) \rho : L^2(\mathbb R^3) \to H^4(\mathbb R^3)
$
 is bounded for any $\rho\in C_c^\infty(\mathbb R^3)$.
\end{theorem}

\begin{proof}
First we consider the case  $R_V(\lambda)$ for $\Re\lambda \gg1$ and $\Im\lambda\gg 1$. Noting the equality 
\begin{align}\label{equality}
(\Delta^2 + V(x) -\lambda^4)R_0(\lambda) = (\Delta^2 - \lambda^4) R_0(\lambda) + V(x)R_0(\lambda) = I + V(x)R_0(\lambda),
\end{align}
and recalling the free resolvent estimate \eqref{spectral1}:
\[
\|R_0(\lambda)\|_{L^2(\mathbb R^3)\rightarrow L^2(\mathbb R^3)}\leq \frac{1}{|\lambda|^2(\Im\lambda)
(\Re\lambda)} \leq (\Im\lambda)^{-2}(\Re\lambda)^{-2}, \quad \lambda \in \Sigma,
\]
we obtain for $\Re\lambda\gg 1$ and $\Im\lambda\gg 1$ that
\begin{align*}
\|VR_0(\lambda)\|_{L^2(\mathbb R^3)\rightarrow L^2(\mathbb R^3)} \leq \|V\|_{L^\infty(\mathbb R^3)}\|R_0(\lambda)\|_{L^2(\mathbb R^3)\rightarrow L^2(\mathbb R^3)}\leq\frac{\|V\|_{L^\infty(\mathbb R^3)}}{(\Im\lambda)^2
(\Re\lambda)^2}\leq\frac{1}{2}.
\end{align*}
Hence the operator $I + VR_0(\lambda)$ is invertible and the Neumann series reads
\[
(I + VR_0(\lambda))^{-1} = \sum_{k=0}^\infty (-1)^k (VR_0(\lambda))^k.
\]
Combining with \eqref{equality} gives that
\[
R_V(\lambda) = R_0(\lambda) (I + VR_0(\lambda))^{-1}
\]
are well-defined bounded operators of $\mathcal B(L^2, L^2)$ for $\Re\lambda\gg 1$ and $\Im\lambda\gg 1$. Moreover, it is easy to see that $VR_0(\lambda)$ is an analytic family of compact operators on $\Sigma$. Consequently, it follows from the analytic Fredholm theorem that $ (I + VR_0(\lambda))^{-1}$ is in fact a meromorphic family of operators for $\lambda\in\Sigma$. which implies that $R_V(\lambda): L^2(\mathbb R^3)\rightarrow L^2(\mathbb R^3)$ is a meromorphic family of operators in $\Sigma$.

Next, we consider the extension of $R_V(\lambda)$ from $\Sigma$ to the whole complex plane $\mathbb C$ as the operator $L_{\rm comp}^2(\mathbb R^3)\rightarrow H^4_{\rm loc}(\mathbb R^3)$. To this end, we define the following meromorphic family of operators:
\[
T(\lambda) = VR_0(\lambda): L_{\rm comp}^2(\mathbb R^3)\rightarrow L^2_{\rm comp}(\mathbb R^3).
\]
Since $V\in L^\infty_{\rm comp}(\mathbb R^3)$ with a compact support, we can choose $\rho\in C_0^\infty(\mathbb R^3)$ such that $\rho(x) = 1$ on $\text{supp} V$. Thus, by checking $\rho T(\lambda) = \rho V R_0(\lambda) = VR_0{\lambda} = T(\lambda)$, we know that $(1 - \rho)T(\lambda) = 0$, 
\[
(I + T(\lambda)(1 - \rho))^{-1} = I - T(\lambda)(1 - \rho)
\]
and
\[
(I + T(\lambda))^{-1} = (I + T(\lambda)\rho)^{-1} (I - T(\lambda)(1 - \rho)).
\]
Therefore
\begin{align}\label{expression1}
R_V(\lambda) = R_0(\lambda)(I + T(\lambda))^{-1} =  R_0(\lambda) (I + T(\lambda)\rho)^{-1} (I - T(\lambda)(1 - \rho)).
\end{align}
Note that
\[
I - T(\lambda)(1 - \rho): L^2_{\rm comp}(\mathbb R^3)\rightarrow L^2_{\rm comp}(\mathbb R^3)
\]
and
\[
R_0(\lambda):  L^2_{\rm comp}(\mathbb R^3)\rightarrow H^4_{\rm loc}(\mathbb R^3)
\]
are both meromorphic for $\lambda \in \mathbb{C}$. Hence in order to obtain the meromorphic extension of $R_V(\lambda)$ to $\mathbb C$, it suffices to prove
\[
(I + T(\lambda)\rho)^{-1} : L^2_{\rm comp}(\mathbb R^3)\rightarrow L^2_{\rm comp}(\mathbb R^3)
\]
is a meromorphic family of operators on $\mathbb C$. Since by $V(x) = V(x)\rho(x)$ we have $T(\lambda)\rho = V\rho R_0(\lambda)\rho$ and 
\begin{align*}
\|T(\lambda)(\rho)\|_{L^2_{\rm comp}\rightarrow L^2_{\rm comp}} &= \|V\rho R_0(\lambda)\|_{L^2_{\rm comp}(\mathbb R^3)\rightarrow L^2_{\rm comp}(\mathbb R^3)}\\
&\leq\|V\|_{L^\infty} \|\rho R_0{\lambda}\rho\|_{L^2_{\rm comp}(\mathbb R^3)\rightarrow L^2_{\rm comp}(\mathbb R^3)}\\
&\leq C|\lambda|^{-2} (e^{2R(\Im\lambda)_-} + e^{2R(\Re\lambda)_-})\\
&\leq\frac{1}{2}
\end{align*}
for  $\Re\lambda\gg 1$ and $\Im\lambda\gg 1$. Hence it follows from the Neumann series the operator $(I + T(\lambda)\rho)^{-1}: L^2(\mathbb R^3)\rightarrow L^2(\mathbb R^3)$ exists for $\Re\lambda\gg 1$ and $\Im\lambda\gg 1$. Moreover, for any $\lambda\in\mathbb{C}\backslash \{0\},$ the operator $T(\lambda)\rho = V\rho R_0(\lambda)\rho$ is compact on $L^2(\mathbb R^3)$ by \eqref{free}. Therefore, it follows from the analytic Fredholm theorem that $(I + T(\lambda)\rho)^{-1}: L^2(\mathbb R^3)\rightarrow L^2(\mathbb R^3)$ is meromorphic on $\mathbb C$.

Finally, it remains to show that $(I + T(\lambda)\rho)^{-1}$ is 
 $L^2_{\rm comp}(\mathbb R^3)\rightarrow L^2_{\rm comp}(\mathbb R^3)$ for all $\lambda\in\mathbb{C}\backslash\{0\}$. In fact, we can choose $\chi, \tilde{\chi}\in C_0^\infty$ such that $\chi\rho = \rho$ and that $\tilde{\chi}\chi = \chi$, then $(1 - \tilde{\chi})\rho = 0$. Moreover, when $\Re\lambda\gg 1, \Im\lambda\gg 1$, by the Neumann series argument and $V\rho = V$, we have
\begin{align}\label{cutoff}
(1 - \tilde{\chi})(I + T(\lambda)\rho)^{-1}\chi &= (1 - \tilde{\chi})\chi + \sum_{k=1}^\infty (-1)^k (1 - \tilde{\chi}) (T(\lambda)\rho)^k\chi\notag\\
&= \sum_{k=1}^\infty (-1)^k (1 - \tilde{\chi}) (V\rho R_0(\lambda)\rho)^k\chi\notag\\
&= \sum_{k=1}^\infty (-1)^k (1 - \tilde{\chi}) (V\rho R_0(\lambda)\rho)(V\rho R_0(\lambda)\rho)^{k-1}\chi\notag\\
&=0,
\end{align}
where the last equality uses $(1 - \tilde{\chi})\rho = 0$. By the analytic continuation, \eqref{cutoff} remains true for all $\lambda\in\mathbb{C}$. Therefore, by the expression \eqref{expression1} of $R_V$ we obtain that $R_V(\lambda)$ is meromorphic of $\lambda$ on $ \mathbb{C}$ as a family of operators $ L^2_{\rm comp}(\mathbb R^3)\rightarrow H^4_{\rm loc}(\mathbb R^3)$, which completes the proof. 
\end{proof}

In the following theorem, we further give a resonance-free region and a resolvent estimate of $\rho R_V(\lambda)\rho: L^2(\mathbb R^3)\rightarrow H^4(\mathbb R^3)$ for a given $\rho\in C_0^\infty(\mathbb R^3)$, which play a crucial role in the  stability analysis for the inverse problem. 

\begin{theorem}\label{bound_2}
Let $V(x)\in L^\infty_{\rm comp}(\mathbb R^3, \mathbb C)$. Then for any given $\rho\in C_0^\infty(\mathbb R^3)$ satisfying $\rho V = V$, i.e., $\text{supp} (V)\subset \text{supp} (\rho)\subset\subset B_R$, there exists a positive constant $C$ depending on $\rho$ and $V$ such that
\begin{align}\label{bound_3}
\|\rho R_V(\lambda)\rho\|_{L^2(B_R)\rightarrow H^j(B_R)}\leq C|\lambda|^{-2 + j} (e^{2R(\Re\lambda)_-}+ 
e^{2R(\Im\lambda)_-}),\quad j = 0, 1, 2, 3, 4,
\end{align}
where $\lambda\in \Omega_\delta$. Here $\Omega_\delta$ denotes the resonance-free region defined as
\begin{align*}
\Omega_\delta:=\Big\{\lambda: {\Im}\lambda&\geq - A - \delta {\rm log}(1 + |\lambda|), \,{\Re}\lambda \geq - A - \delta {\rm log}(1 + |\lambda|), |\lambda|\geq C_0\Big\},
\end{align*}
where $A$ and $C_0$ are two positive constants and $\delta$ satisfies $0<\delta<\frac{1}{2R}$. 
\end{theorem}

\begin{proof}
By the estimate \eqref{free2} we obtain 
\begin{align}\label{free3}
\|\rho_1 R_0(\lambda)\rho_1\|_{L^2(\mathbb R^3)\to L^2(\mathbb R^3)} \leq C |\lambda|^{-2} (e^{2R(\Re\lambda)_-}+ e^{2R(\Im\lambda)_-}).
\end{align}
Choosing $\rho_1\in L_{\rm comp}^\infty(\mathbb R^3)$  such that $\rho_1V = V$ (e.g., $\rho_1 (x)= \chi_{{\rm supp}V}(x)$),  we have
\begin{align*}
\|V R_0(\lambda)\rho_1\|_{L^2(\mathbb R^3)\rightarrow L^2(\mathbb R^3)} &= \|V\rho_1 R_0(\lambda)\rho_1\|_{L^2(\mathbb R^3)\rightarrow L^2(\mathbb R^3)}\\
&\lesssim \|V\|_{L^\infty(\mathbb R^3)}|\lambda|^{-2} (e^{2R(\Re\lambda)_-}+ e^{2R(\Im\lambda)_-})\\
&\lesssim \|V\|_{L^\infty(\mathbb R^3)}|\lambda|^{-2}e^{2R(A + \delta\text{log}(1 + |\lambda|))}\\
&\lesssim  \|V\|_{L^\infty(\mathbb R^3)}|\lambda|^{-2}\leq\frac{1}{2}
\end{align*}
provided that $\lambda \in \Omega_\delta$ with $\Omega_\delta$ being defined by
\begin{align*}
\Omega_\delta:=\Big\{\lambda: \Im\lambda&\geq - A - \delta {\rm log}(1 + |\lambda|), \,\Re\lambda \geq - A - \delta {\rm log}(1 + |\lambda|), \, |\lambda|\geq C_0\Big\},
\end{align*}
which is shown in Figure \ref{dR}. Here we let $A$ be a positive constant, $C_0\gg 1$ and $\delta<\frac{1}{2R}.$
Hence by the Neumann series argument we can prove that the inverse operator $(I + VR_0(\lambda)\rho_1)^{-1}$ exists for all $\lambda\in\Omega_\delta$, and
\begin{align}\label{free4}
\|(I + VR_0(\lambda)\rho_1)^{-1}\|_{L^2(\mathbb R^3)\rightarrow L^2(\mathbb R^3)} = \|(I + V\rho_1R_0(\lambda)\rho_1)^{-1}\|_{L^2(\mathbb R^3)\rightarrow L^2(\mathbb R^3)}\leq 2.
\end{align}

\begin{figure}
\centering
\includegraphics[width=0.8\textwidth]{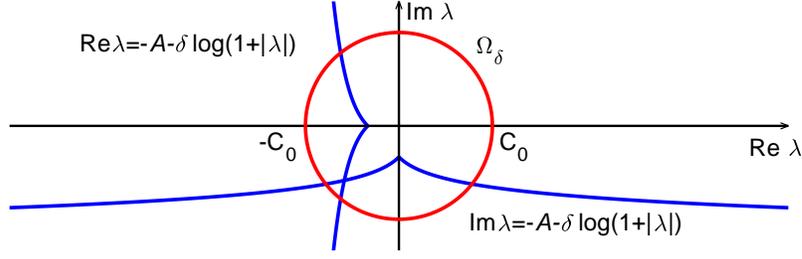}
\caption{The resonance-free region $\Omega_\delta$.}
\label{dR}
\end{figure}

Now let $\rho\in C_0^\infty(\mathbb R^3)$ such that $\rho V = V$. Then we have 
\[
\rho R_V(\lambda) \rho = \rho R_V(\lambda) \rho (I + VR_0(\lambda)\rho_1)^{-1} (I - VR_0(\lambda)(1 -\rho_1))\rho.
\]
Consequently, combining  \eqref{free3} and \eqref{free4} we obtain the desired estimate for $j = 0$ as
\[
\|\rho R_V(\lambda)\rho\|_{L^2(\mathbb R^3)\rightarrow L^2(\mathbb R^3)} \leq C |\lambda|^{-2}
\big(e^{2R(\Re\lambda)_-}+ e^{2R(\Im\lambda)_-}\big).
\]
 For the case $j = 4$, let $\tilde{\rho}\in C_0^\infty(\mathbb R^3)$ such that $\tilde{\rho} = 1$ on $\text{supp}\rho$ and $\text{supp}\tilde{\rho} \subset B_R$. Then it holds that
\begin{align*}
\|\rho R_V\rho f\|_{H^4(B_R)} &\lesssim  \big(\|\tilde{\rho} R_V(\lambda)\rho f\|_{L^2(B_R)} + \|\tilde{\rho} \Delta^2 R_V(\lambda)\rho f\|_{L^2(B_R)}\big)\\
&\lesssim \big(\|\tilde{\rho} R_V(\lambda)\rho f\|_{L^2(B_R)} + \|\tilde{\rho} H R_V(\lambda)\rho f\|_{L^2(B_R)} + \|\tilde{\rho} V R_V(\lambda)\rho f\|_{L^2(B_R)}\big)\\
&\lesssim  (1+ \lambda^2)^2  \|\tilde{\rho} R_V(\lambda)\rho f\|_{L^2(B_R)} \\
&\lesssim (1+ \lambda^2)^2 |\lambda|^{-2} \big(e^{2R(\Re\lambda)_-}+ e^{2R(\Im\lambda)_-}\big)\|f\|_{L^2(B_R)}
\end{align*}
for $\lambda\in\Omega_\delta$. Finally, the cases of $j =1, 2, 3$ follows by an application of the interpolation between $j = 0$ and $j = 4$. 
\end{proof}

By Theorem \ref{bound_2}, the scattering problem \eqref{main_eq}--\eqref{src} has a unique solution for all positive wavenumbers $\kappa\geq C_0$, which is stated below. 

\begin{corollary}
Let $V(x)\in L^\infty_{\rm comp}(\mathbb R^3, \mathbb C)$. For all positive wavenumbers $\kappa\geq C_0$ where $C_0$ is specified in Theorem \ref{bound_2}, the scattering problem \eqref{main_eq}--\eqref{src} admits a unique solution $u\in H^4(B_R)$ such that 
\[
\|u\|_{H^4(B_R)}\lesssim \|f\|_{L^2(B_R)}.
\]
\end{corollary}

\subsection{Inverse problem}

In this section, we discuss the uniqueness and stability of the inverse problem for the general case with a nontrivial potential $V(x)\geq 0$.

Again, we begin with the the spectrum of the operator $H$ with the Navier boundary condition. Let $\{\mu_j, \phi_j\}_{j=1}^\infty$ be the positive increasing
eigenvalues and eigenfunctions of $H$ in $B_R$, where $\phi_j$ and $\mu_j$
satisfy
\[
\begin{cases}
(\Delta^2 + V(x)) \phi_j(x)=\mu_j\phi_j(x)&\quad\text{in }B_R,\\
\Delta \phi_j(x) = \phi_j(x)=0&\quad\text{on }\partial B_R.
\end{cases}
\]
Assume that $\phi_j$ is normalized such that
\[
\int_{B_R}|\phi_j(x)|^2{\rm d}x=1.
\] 
Consequently, we obtain the spectral decomposition of $f$:
\[
f(x)=\sum_{j=1}^\infty f_j\phi_j(x),
\]
where
\[
f_j=\int_{B_R}f(x)\bar{\phi}_j(x){\rm
d}x.
\]
It is clear that 
\[
 \|f\|^2_{L^2(B_R)} = \sum_{j} |f_j|^2.
\]

The following lemma gives a link between the values of an analytical function for small and large arguments. The proof can be found in \cite[Lemma A.1]{LZZ}. 

\begin{lemma}\label{ac}
Let $p(z)$ be analytic in the infinite rectangular slab
\[
D = \{z\in \mathbb C: (\tilde{A}, +\infty)\times (-d, d) \}, 
\]
where $\tilde{A}$ is a positive constant, and continuous in $\overline{D}$ satisfying
\begin{align*}
\begin{cases}
|p(z)|\leq \tilde{\epsilon}, &\quad z\in (\tilde{A}, \tilde{A}_1],\\
|p(z)|\leq M, &\quad z\in R,
\end{cases}
\end{align*}
where $\tilde{A}, \tilde{A}_1, \tilde{\epsilon}$ and $M$ are positive constants. Then there exists a function $\eta(z)$ with $z\in (\tilde{A}_1, +\infty)$ satisfying 
\begin{equation*}
\eta(z) \geq \frac{64ad}{3\pi^2(a^2 + 4d^2)} e^{\frac{\pi}{2d}(\frac{a}{2} - z)},
\end{equation*}
where $a = \tilde{A}_1 - \tilde{A}$, such that
\begin{align*}
|p(z)|\leq M\tilde{\epsilon}^{\eta(z)}\quad \forall\, z\in (\tilde{A}_1, +\infty).
\end{align*}
\end{lemma}

The following lemma gives an estimate for the normal derivatives of the eigenfunctions on $\partial B_R$ and a Weyl-type inequality for the Dirichlet eigenvalues.

\begin{lemma}\label{eigenfunction_est1}
The following estimate holds:
\begin{align}\label{boundary_estimate_2}
\|\partial_\nu \phi_j\|_{L^2(\partial B_R)}\leq C\kappa^2_j,\quad \|\partial_\nu (\Delta\phi_j)\|_{L^2(\partial B_R)}\leq C\kappa^4_j,
\end{align}
where the positive constant $C$ is independent of $j$. Moreover, the following Weyl-type inequality holds for the Dirichlet eigenvalues $\{\mu_n\}_{n=1}^\infty$:
\begin{align}\label{weyl_1}
E_1 n^{4/3}\leq \mu_n\leq E_2 n^{4/3},
\end{align}
where $E_1$ and $E_2$ are two positive constants independent of $n$.
\end{lemma}

\begin{proof}
We begin with the estimate \eqref{boundary_estimate_2} for the eigenfunctions on the boundary. Let $u$ be an eigenfunction  with eigenvalue $\mu$ such that
\begin{align*}
\begin{cases}
H u = \mu u,\quad &x\in B_R,\\
u = \Delta u = 0, \quad &x\in\partial B_R.
\end{cases}
\end{align*}
Define a differential operator 
\[
A = \frac{1}{2}(x\cdot \nabla+ \nabla\cdot x) = x\cdot\nabla + \frac{3}{2} = |x|\partial_\nu + \frac{3}{2}.
\]
Denote the commutator of two differential operators by $[\cdot, \, \cdot]$ such that $[O_1, O_2] = O_1 O_2 - O_2 O_1$ for two differential operators $O_1$ and $O_2$. Then we have  
\begin{align}\label{com}
[\Delta^k, A] = 2k\Delta^k, \quad k\in\mathbb N^+.
\end{align}
Denote $B = A\Delta$. A simple calculation gives
\begin{align*}
&\int_{B_R} u [H, B] u {\rm d}x
= \int_{B_R} \left(u (\Delta^2 + V) (Bu) - u B (\Delta^2 + V) u \right){\rm d}x\\
&= \int_{B_R} (\Delta^2 u + Vu - \mu u) Bu {\rm d}x+ \int_{\partial B_R}\left( u\partial_\nu(\Delta Bu) - \partial_\nu u\Delta(Bu) \right) {\rm d}s\\
&\quad + \int_{\partial B_R}\left( \Delta u\partial_\nu(Bu) - \partial_\nu (\Delta u)Bu \right) {\rm d}s\\
&= - \int_{\partial B_R} \left(\partial_\nu u\Delta(Bu) + \partial_\nu (\Delta u)Bu\right) {\rm d}s\\
&= - \int_{\partial B_R}\left( \partial_\nu u\Delta(Bu) + R|\partial_\nu (\Delta u)|^2\right){\rm d}s,
\end{align*}
where we have used $u = \Delta u = 0$ on $\partial B_R$ and Green's formula. Since by \eqref{com} we have 
\[
\Delta Bu = \Delta A\Delta = (A\Delta + 2\Delta)\Delta = A\Delta^2 + 2\Delta^2,
\]
It holds that 
\begin{align*}
&\int_{\partial B_R} \partial_\nu u\Delta(Bu){\rm d}s = \int_{\partial B_R} \partial_\nu u (A\Delta^2 + 2\Delta^2)u {\rm d}s \\
&=  \int_{\partial B_R} \Big(\partial_\nu u \big(\big(R\partial_\nu + 
\frac{3}{2}\big)\Delta^2 u\big) + 2\Delta^2u \Big){\rm d}s\\
&= R\int_{\partial B_R} \partial_\nu u \,\partial_\nu(\Delta^2 u){\rm d}s
= R\int_{\partial B_R} \partial_\nu u \,\partial_\nu(\mu u - Vu){\rm d}s,
\end{align*}
where we have used $\Delta^2u = -Vu + \mu u = 0$ and $u = 0$ on $\partial B_R$. Hence we have
\begin{align}\label{I}
 \Big|\int_{\partial B_R} \partial_\nu u\Delta(Bu){\rm d}s\Big| \geq (\mu - \|V\|_{L^\infty(B_R)}) \int_{\partial B_R} |\partial_\nu u|^2{\rm d}s.
\end{align}
On the other hand we have 
\begin{align}\label{II}
\int_{\partial B_R} \partial_\nu (\Delta u)Bu {\rm d}s= \int_{\partial B_R} \partial_\nu (\Delta u)Bu {\rm d}s = R \int_{\partial B_R} |\partial_\nu (\Delta u)|^2 {\rm d}s.
\end{align}
Moreover, since by \eqref{com} it holds that $[H, B] = 4\Delta^3 + [V, A\Delta]$, we obtain 
\begin{align}\label{III}
&\Big|\int_{B_R} u [H, B] u {\rm d}x\Big|= \Big|\int_{B_R} \left(4u\Delta^3u + [V, A\Delta]u\right) {\rm d}x\Big|\notag\\
&= \Big|\int_{B_R}\left( 4u\Delta (-Vu + \mu u) + [V, A\Delta]u \right){\rm d}x\Big|\notag\\
&\leq C\mu \|u\|^2_{H^2(B_R)}\leq C\mu^2.
\end{align}
Here we have used the fact that the commutator $[V, A\Delta]$ has order of 2 at most. Using \eqref{I}--\eqref{III} we obtain 
\[
\|\partial_\nu u\|^2_{L^2(\partial B_R)}\leq\mu,\quad \|\partial_\nu (\Delta u)\|^2_{L^2(\partial B_R)}\leq\mu^2,
\]
which completes the proof of \eqref{boundary_estimate_2}.

Next, we prove the Weyl-type inequality \eqref{weyl_1}. Assume $\mu_1<\mu_2<\cdots$ are the eigenvalues of the
operator $H$. Then we have following min-max principle: 
\[
\mu_n=\max_{\phi_1,\cdots,\phi_{n-1}}\min_{\psi\in[\phi_1,\cdots,
\phi_{n-1}]^\perp\atop \psi\in H_0^2(B_R)}\frac{\int_{B_R} |\Delta
\psi|^2 + V|\psi|^2{\rm d}x}{\int_{B_R}\psi^2{\rm d}x}.
\]
Assume that $\mu_1^{(1)}<\mu_2^{(1)}<\cdots$ are the eigenvalues for the operator
$\Delta^2$. By the min-max principle, we have
\[
C_1\mu_n^{(1)}<\mu_n<C_2\mu_n^{(1)}, \quad n=1, 2, \dots,
\]
where $C_1$ and $C_2$ are two positive constants depending on $\|V\|_{L^\infty(B_R)}$. We have from Weyl's law \cite{Weyl} for $\Delta^2$ that 
\[
\lim_{n\rightarrow+\infty}\frac{\mu_n^{(1)}}{n^{4/3}}=D,
\]
where $D$ is a constant. Therefore there exist two constants $E_1$ and $E_2$
such that 
\[
E_1 n^{4/3}\leq \mu_n\leq E_2 n^{4/3},
\]
which completes the proof.
\end{proof}

Denote $\kappa_j^4 = \mu_j$. Let $u(x, \kappa_j)$ be the solution to \eqref{main_eq}--\eqref{src} with $\kappa = \kappa_j$. In general, due to the presence of the potential function  $V(x)$, the resolvent $R_V(\kappa)$ may have poles on $\mathbb R^+$ which are restricted in the set $\{x: 0<x<C_0\}$ by Theorem \ref{bound_2}. However, a recent result \cite{Yao} shows that for a certain class of nonnegative potential functions the resolvent has no poles on $\mathbb R^+$. On the other hand, since the first eigenvalue $\mu_1$ of $H$ in $B_R$ increases as the radius $R$ decreases, if the supports of the source $f(x)$ and potential $V(x)$ are small, we may shrink the ball $B_R$ to make $\mu_1\geq C_0$. Thus, in the following study of the inverse problem, we assume that the resolvent $R_V(\kappa)$ has no poles in the set $\{\kappa_j: \kappa_j<C_0\}$.

\begin{lemma}\label{fj_1}
The following estimate holds:
\[
|f_j|^2 \lesssim \kappa_j^4 \|u(x,\kappa_j)\|^2_{L^2(\partial B_R)}  + \kappa_j^8 \|\Delta u(x,\kappa_j)\|^2_{L^2(\partial B_R)}
\]
for $j=1, 2, 3, \dots$.
\end{lemma}

\begin{proof}
Multiplying both sides of \eqref{main_eq} by $\bar{\phi}_j$ and using
the integration by parts yield
\begin{align*}
\int_{B_R} f(x)\bar{\phi}_j(x) {\rm d}x &= \int_{\partial B_R}\left(\partial_\nu (\Delta u(x, \kappa_j)) \bar{\phi}_j - \Delta u(x, \kappa_j)\partial_\nu \bar{\phi}_j\right){\rm d}s\\
 &\quad + \int_{\partial B_R}\left( \partial_\nu u(x, \kappa_j) \Delta \bar{\phi}_j - u(x, \kappa_j)\partial_\nu (\Delta \bar{\phi}_j)\right){\rm d}s.
\end{align*}
Moreover, noting that $\Delta \bar{\phi}_j =  \bar{\phi}_j = 0$ we arrive at 
\begin{align*}
\int_{B_R} f(x)\bar{\phi}_j(x) {\rm d}x = -\int_{\partial B_R} \left( \Delta u(x, \kappa_j)\partial_\nu \bar{\phi}_j
 + u(x, \kappa_j)\partial_\nu (\Delta \bar{\phi}_j)\right){\rm d}s.
\end{align*}
The proof is completed by using Lemma \ref{eigenfunction_est1} and the Schwartz inequality.
\end{proof}

\begin{lemma}\label{ac_2}
Let $f$ be a real-valued function and $\|f\|_{L^2(B_R)}\leq Q$. There exist positive constant $d$
and positive constants $\tilde{A},\tilde{A}_1$ satisfying $C_0< \tilde{A}<\tilde{A}_1$, which do not depend on $f$ and $Q$, such that 
\[
\kappa^4\|u(x,\kappa)\|^2_{L^2(\partial B_R)} + \kappa^8\|\Delta u(x,\kappa)\|^2_{L^2(\partial B_R)}\lesssim
Q^2e^{6R\kappa}\tilde{\epsilon}_1^{2\eta(\kappa)} \quad \forall\, \kappa\in (\tilde{A}_1, +\infty).
\]
Here $C_0$ is specified in Theorem \ref{bound_2} and
\begin{align*}
&\tilde{\epsilon}^2_1 :=\kappa^4\|u(x,\kappa)\|^2_{L^2(\partial B_R)} + \kappa^8\|\Delta u(x,\kappa)\|^2_{L^2(\partial B_R)},\\ 
&\eta(\kappa) \geq \frac{64ad}{3\pi^2(a^2 + 4d^2)} e^{\frac{\pi}{2d}(\frac{a}{2} -
\kappa)}.
\end{align*}
Here $a = \tilde{A}_1 - \tilde{A}.$
\end{lemma}

\begin{proof}

Let 
\[
I(\kappa):=\int_{\partial
B_R}\left(\kappa^4 u(x,\kappa)u(x, {\rm i}\kappa) + \kappa^8 \Delta u(x,\kappa) \Delta u(x, {\rm i}\kappa)\right) {\rm d}s,\quad \kappa\in \mathbb C.
\]
By Theorem \ref{bound_2} we have that $u(x, \kappa)$ is analytic in the domain 
\begin{align*}
\Omega_\delta:=\Big\{\lambda: \Im\lambda\geq -A -\delta {\rm log}(1 + |\lambda|), \,\Re\lambda \geq -A - \delta {\rm log}(1 + |\lambda|), |\lambda|\geq |C_0|\Big\}.
\end{align*} 
Moreover, there exist $d>0$ and $\tilde{A}>C_0$ such
that $\mathcal{R} = (\tilde{A}, +\infty)\times (-d, d)\subset\Omega_\delta$ and  ${\rm i}\mathcal{R} =\{{\rm i}z: z\in\mathcal{R}\}\subset\Omega_\delta$. The geometry of the domain $\mathcal{R}$ is shown in Figure \ref{dR-1}.
\begin{figure}
\centering
\includegraphics[width=0.8\textwidth]{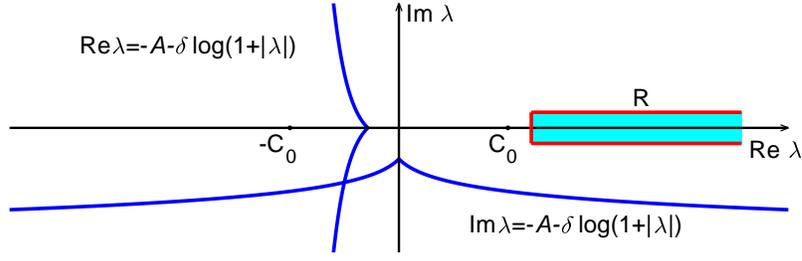}
\caption{The region $\mathcal{R}$.}
\label{dR-1}
\end{figure}
Therefore, $I(\kappa)$
is analytic in $\mathcal{R}$. Moreover, since $f(x)$ and $V(x)$ are both real-valued functions,  we have $\overline{u(x,\kappa)}=u(x,{\rm i}\kappa)$ and  $\overline{\Delta u(x,\kappa)}=\Delta u(x,{\rm i}\kappa)$ for $\kappa\geq C_0$. Hence we have
\[
I(\kappa)=\kappa^4\|u(x,\kappa)\|^2_{L^2(\partial B_R)} + \kappa^8\|\Delta u(x,\kappa)\|^2_{L^2(\partial B_R)}, \quad \kappa\geq C_0.
\]
By the estimate \eqref{bound_3} in Theorem \ref{bound_2} we have for $\kappa\in \mathcal{R} $ that 
\begin{align*}
&|\kappa|^2\|u(x,\kappa)\|_{L^2(\partial B_R)} + |\kappa|^4 \|\Delta u(x,\kappa)\|_{L^2(\partial B_R)}\\
&\leq |\kappa|^2\|u(x,\kappa)\|_{H^{1/2}(\partial B_R)} + |\kappa|^4 \|\Delta u(x,\kappa)\|_{H^{3/2}(\partial B_R)}\\
&\leq 2|\kappa|^4\|u\|_{H^4(\mathbb R^3)}\leq  e^{3R|\kappa|}\|f\|_{L^2(B_R)}.
\end{align*}
Since
\begin{align*}
|I(\kappa)|&\leq |\kappa|^2 \|u(x,\kappa)\|_{L^2(\partial B_R)} |\kappa|^2
\|u(x,-\kappa)\|_{L^2(\partial B_R)}  \\
&\quad + |\kappa|^4 \|\Delta u(x,\kappa)\|_{L^2(\partial B_R)} |\kappa|^4
\|\Delta u(x,-\kappa)\|_{L^2(\partial B_R)}\\
&\lesssim e^{6R|\kappa|}\|f\|^2_{L^2(B_R)},\quad\kappa\in \mathcal{R} ,
\end{align*}
we have 
\[
|e^{-6R|\kappa|}I(\kappa)|\lesssim Q^2,\quad\kappa\in \mathcal{R}.
\]
An application of Lemma \ref{ac} shows that there exists a function $\eta(\kappa)$ such that
\[
\big| e^{-6R\kappa} I(\kappa)\big| \lesssim
Q^2\tilde{\epsilon}_1^{2\eta(\kappa)}\quad \forall\, \kappa\in (\tilde{A}_1, +\infty),
\]
where 
\[
\eta(\kappa) \geq \frac{64ad}{3\pi^2(a^2 + 4d^2)} e^{\frac{\pi}{2d}(\frac{a}{2} - \kappa)},
\]
which completes the proof. 
\end{proof}

Here we state the uniqueness result for the inverse problem.

\begin{theorem}
Let $f\in L^2(B_R)$ and $I:=(C_0, C_0 + \zeta) \subset \mathbb R^+$ be an
open interval, where $C_0$ is the constant given  in the definition of
$\widetilde{\Omega}_\delta$ in Lemma \ref{ac_2} and $\zeta$ is any positive
constant. Then the source term $f$ can be uniquely determined by the
multifrequency data $\{u(x,\kappa): x\in \partial B_R, \kappa \in I\}\cup
\{u(x,\kappa_j): x\in\partial B_R, \kappa_j\in (0, C_0]\}$.
\end{theorem}

\begin{proof}
Let $u(x,\kappa)=0$ for $x\in\partial B_R$ and $\kappa \in I \cup \{\kappa_j:
\kappa_j\in (0, C_0])\}$. It suffices to show that $f(x)=0$. Since $u(x,\kappa)$ and  $\Delta u(x,\kappa)$
are both analytic in $\widetilde{\Omega}_\delta$ for $x\in\partial B_R$, it holds that
$u(x,\kappa) = \Delta u(x,\kappa) = 0$ for all eigenvalues $\kappa>C_0$. Then we have that
$u(x,\kappa_j)=0$ for all $\kappa_j, j=1, 2, 3, \cdots$.
Hence, it follows from Lemma \ref{fj_1} that 
\[
\int_{B_R} f(x)\bar{\phi}_j(x) {\rm d}x = 0, \quad j=1, 2, 3, \cdots,
\]
which implies $f = 0$.
\end{proof}

The following lemma is important in the stability analysis.

\begin{lemma}\label{tail_2}
Let $f\in H^{n+1}(B_R)$ and $\|f\|_{H^{n+1}(B_R)}\leq Q$. It holds that 
\[
\sum_{j \geq s} |f_j|^2 \lesssim \frac{Q^2}{s^{\frac{2}{3}(n+1)}}.
\]
\end{lemma}

\begin{proof}
A simple calculation yields 
\[
\sum_{j \geq s} |f_j|^2 \leq \sum_{j \geq s}
\frac{\kappa_j^{2n+2}}{\kappa_s^{2n+2}} |f_j|^2 \leq
\frac{1}{\kappa_s^{2n+2}}\sum_{j \geq s} \kappa_j^{2n+2} |f_j|^2  \lesssim
\frac{Q^2}{\kappa_s^{2n+2}}.
\]
Noting 
\[
\|f\|^2_{H^s(B_R)} \cong\sum_{j=1}^\infty (\kappa_j^{2}+1)^s |f_j|^2, 
\]
and using the Weyl-type inequality in Lemma \ref{eigenfunction_est1}, we have $\kappa_s^2 \geq
E_2 s^{\frac{2}{3}}$ and complete the proof. 
\end{proof}

Define a real-valued function space
\[
\mathcal C_Q = \{f \in H^{n+1}(B_R):  \|f\|_{H^{n+1}(B_R)}\leq Q, ~ \text{supp}
f\subset B_R, ~ f: B_R \rightarrow \mathbb R \}.
\]
Now we are in the position to discuss the inverse source problem for the biharmonic operator with a general potential. Let $f\in \mathcal C_Q$. The inverse source problem is to determine $f$ from the boundary
data $u(x,\kappa)$, $x\in\partial B_R$, $\kappa \in (\tilde{A},\tilde{A}_1) \cup \cup_{j=1}^N
\kappa_j$, where $1\leq N \in \mathbb N$ and $\kappa_N>\tilde{A}_1$, $\tilde{A}$ and $\tilde{A}_1$
are the constants specified in Lemma \ref{ac_2}.

The following stability estimate is the main result of this paper.

\begin{theorem}
Let $u(x,\kappa)$ be the solution of the scattering problem \eqref{main_eq}--\eqref{src} corresponding to the source $f\in \mathcal C_Q$. Then for  sufficiently small $\tilde{\epsilon}_1$, the following estimate holds: 
\begin{align}\label{stability}
\|f\|_{L^2( B_R)}^2  \lesssim \tilde{\epsilon}^2 +
\frac{Q^2}{N^{\frac{1}{3}(n+1)}(\ln|\ln\tilde{\epsilon}_1|)^{\frac{1}{3}(n+1)}},
\end{align}
where
\begin{align*}
\tilde{\epsilon}^2 &= \sum_{j=1}^N \kappa_j^4 \|u(x,\kappa_j)\|^2_{L^2(\partial
B_R)} + \kappa_j^8 \|\Delta u(x,\kappa_j)\|^2_{L^2(\partial
B_R)},\\
\tilde{\epsilon}^2_1 &= {\rm sup}_{\kappa \in (\tilde{A},\tilde{A}_1)} \kappa^4 \|u(x,\kappa)\|^2_{L^2(\partial
B_R)} + \kappa^8 \|\Delta u(x,\kappa_j)\|^2_{L^2(\partial
B_R)}.
\end{align*}
\end{theorem}

\begin{proof}
We can assume that $\tilde{\epsilon}_1 \leq e^{-1}$, otherwise the estimate is obvious. First, we link the data $\kappa^4 \|u(x,\kappa)\|^2_{L^2(\partial B_R)} + \kappa^8 \|\Delta u(x,\kappa)\|^2_{L^2(\partial B_R)}$
for large wavenumber $\kappa$ satisfying $\kappa\leq \tilde{L}$ with the given data
$\tilde{\epsilon}_1$ of small wavenumber by using the analytic continuation in Lemma
\ref{ac_2}, where $\tilde{L}$ is some large positive integer to be determined later.  By
Lemma \ref{ac_2}, we obtain 
\begin{align*}
&\kappa^4\|u(x,\kappa)\|^2_{\tilde{L}^2(\partial B_R)} + \kappa^8\|\Delta u(x,\kappa)\|^2_{\tilde{L}^2(\partial B_R)}
\lesssim Q^2e^{6R|\kappa|} \tilde{\epsilon}_1^{\eta(\kappa)}\\
&\lesssim Q^2{\rm exp}\{6R\kappa - \frac{c_2a}{a^2 + c_3}e^{c_1(\frac{a}{2} - \kappa)}
|{\ln}\tilde{\epsilon}_1|\}\\
&\lesssim  Q^2{\rm exp} \{  - \frac{c_2a}{a^2 + c_3}e^{c_1(\frac{a}{2} - \kappa)}|{\ln}\tilde{\epsilon}_1| (1 -  \frac{c_4\kappa(a^2 + c_3)}{a} e^{c_1(\kappa - \frac{a}{2})}|{\ln}\tilde{\epsilon}_1|^{-1})\}\\
&\lesssim Q^2{\rm exp} \{  - \frac{c_2a}{a^2 + c_3}e^{c_1(\frac{a}{2} - \tilde{L})}|{\ln}\tilde{\epsilon}_1| (1 -  \frac{c_4\tilde{L}(a^2 + c_3)}{a} e^{c_1(\tilde{L} - \frac{a}{2})}|{\ln}\tilde{\epsilon}_1|^{-1})\}\\
&\lesssim Q^2{\rm exp} \{  - b_0e^{-c_1\tilde{L}}|{\ln}\tilde{\epsilon}_1| (1 - b_1\tilde{L} e^{c_1\tilde{L} }|{\ln}\tilde{\epsilon}_1|^{-1})\},
\end{align*}
where $c, c_i, i=1,2$ and $b_0, b_1$ are constants. Let
\begin{align*}
\tilde{L} = 
\begin{cases}
\left[\frac{1}{2c_1}\ln|\ln \tilde{\epsilon}_1|\right], &\quad N\leq \frac{1}{2c_1} \ln|\ln\tilde{\epsilon}_1|,\\
N, &\quad N>  \frac{1}{2c_1}\ln|\ln\tilde{\epsilon}_1|.
\end{cases}
\end{align*}

If $N\leq  \frac{1}{2c_1}\ln|\ln\tilde{\epsilon}_1|$, we obtain for $\tilde{\epsilon}_1$
sufficiently small that 
\begin{align*}
\kappa^4\|u(x,\kappa)\|^2_{L^2(\partial B_R)} + \kappa^8\|\Delta u(x,\kappa)\|^2_{L^2(\partial B_R)}&\lesssim Q^2{\rm exp} \{  - b_0e^{-c_1\tilde{L}}|{\ln}\tilde{\epsilon}_1| (1 - b_1\tilde{L} e^{c_1\tilde{L} }|{\ln}\tilde{\epsilon}_1|^{-1})\}\\
& \lesssim Q^2\exp\{-\frac{1}{2}b_0e^{-c_1\tilde{L}}|\ln \tilde{\epsilon}_1|\}.
\end{align*}
Noting $e^{-x}\leq \frac{(2n+3)!}{x^{2n+3}}$ for $x>0$, we obtain
\begin{align*}
\sum_{j=N+1}^{\tilde{L}}\Big(\kappa_j^4\|u(x,\kappa_j)\|^2_{L^2(\partial B_R)} + \kappa_j^8\|\Delta u(x,\kappa_j)\|^2_{L^2(\partial B_R)}\Big) \lesssim Q^2
\tilde{L}e^{(2n+3)c_1\tilde{L}}|\ln\tilde{\epsilon}_1|^{-(2n+3)}.
\end{align*}

Taking $\tilde{L}=\frac{1}{2c_1}\ln|\ln\tilde{\epsilon}_1|$, combining the above estimates
and  Lemma \ref{tail_2}, we get
\begin{align*}
&\|f\|_{L^2( B_R)}^2\lesssim \sum_{j=1}^N |f_j|^2+\sum_{j=N+1}^{\tilde{L}} |f_j|^2 + \sum_{j=\tilde{L}+1}^{+\infty} |f_j|^2\\
&\lesssim \sum_{j=1}^N \Big(\kappa_j^4\|u(x,\kappa_j)\|^2_{L^2(\partial B_R)} + \kappa_j^8\|\Delta u(x,\kappa_j)\|^2_{L^2(\partial B_R)} \Big)\\
&\quad +\sum_{j=N+1}^{\tilde{L}}  \Big(\kappa_j^4\|u(x,\kappa_j)\|^2_{L^2(\partial B_R)} + \kappa_j^8\|\Delta u(x,\kappa_j)\|^2_{L^2(\partial B_R)}\Big)+ \frac{1}{\tilde{L}^{\frac{2}{3}(n+1)}}\|f\|^2_{H^{n+1}(B_R)}\\
&\lesssim \tilde{\epsilon}^2 + \tilde{L}Q^2e^{(2n+3)c_1\tilde{L}}|\ln\tilde{\epsilon}_1|^{-(2n+3)}+ \frac{Q^2}{\tilde{L}^{\frac{2}{3}(n+1)}}\\
& \lesssim \tilde{\epsilon}^2 + Q^2\left((\ln|\ln\tilde{\epsilon}_1|)|\ln\tilde{\epsilon}_1|^{\frac{2n+3}{2}}|\ln\tilde{\epsilon}_1|^{-(2n+3)}+(\ln|\ln\tilde{\epsilon}_1|)^{-\frac{2}{3}(n+1)}\right)\\
& \lesssim \tilde{\epsilon}^2 + Q^2\left((\ln|\ln\tilde{\epsilon}_1|)|\ln\tilde{\epsilon}_1|^{-\frac{2n+3}{2}}+(\ln|\ln\tilde{\epsilon}_1|)^{-\frac{2}{3}(n+1)}\right)\\
& \lesssim \tilde{\epsilon}^2 + Q^2(\ln|\ln\tilde{\epsilon}_1|)^{-\frac{2}{3}(n+1)}\\
& \lesssim \tilde{\epsilon}^2 +
\frac{Q^2}{N^{\frac{1}{3}(n+1)}(\ln|\ln\tilde{\epsilon}_1|)^{\frac{1}{3}(n+1)}},
\end{align*}
where we have used $|\ln\tilde{\epsilon}_1|^{1/2}\geq \ln|\ln\tilde{\epsilon}_1|$ for
sufficiently small $\tilde{\epsilon}_1$.

If $N>  \frac{1}{2c_1}\ln|\ln\tilde{\epsilon}_1|$, we have from  Lemma \ref{tail} that 
\begin{align*}
\|f\|_{L^2( B_R)}^2 &\lesssim \sum_{j=1}^N |f_j|^2+ \sum_{j=N+1}^{+\infty} |f_j|^2
\lesssim\tilde{\epsilon}^2 + \frac{Q^2}{N^{\frac{2}{3}(n+1)}}\\
&\lesssim\tilde{\epsilon}^2 +
\frac{Q^2}{N^{\frac{1}{3}(n+1)}(\ln|\ln\tilde{\epsilon}_1|)^{\frac{1}{3}(n+1)}},
\end{align*}
which completes the proof.
\end{proof}

It is also clear to note that the stability \eqref{stability} consists of two parts: the data discrepancy and
the high wavenumber tail. The former is of the Lipschitz type. The latter
decreases as $N$ increases which makes the problem have an almost Lipschitz
stability. The result reveals that the problem becomes more stable when higher
wavenumber data is used. Compared with \eqref{stability_1}, the stability \eqref{stability} has a 
double logarithmic type high frequency tail, which makes the problem more ill-posed. The reason is that, in the presence of the zeroth order
perturbation, the resonance-free region obtained in Theorem \ref{bound_2} is not as good as that obtained in Theorem \ref{free_estimate}.

\section{Conclusion}\label{con}

We have presented stability results on the inverse source problem for the biharmonic operators without and with a zeroth order perturbation. The analysis requires the Dirichlet data only at multiple discrete wavenumbers. The increasing
stability is achieved to reconstruct the source term, and it consists of the data discrepancy and the high wavenumber tail of the source function. The result shows that the ill-posedness of the inverse source problem decreases as the wavenumber increases for the data. A possible continuation of this work is to extend the stability to the case of poly-harmonic operators.  A related but more challenging problem is to study the stability of the inverse potential problem which is to determine $V$.

\end{document}